\documentclass[]{amsart}

\usepackage[utf8]{inputenc}
\usepackage{amssymb}
\usepackage{amsthm}
\usepackage{amsmath}
\usepackage{tikz}

\usepackage[toc,page]{appendix}

\usepackage{enumitem}
\usepackage{nicefrac}

\usepackage{tikz-cd}

\usepackage{tikz-cd}

\usepackage{tikz}
\usetikzlibrary{datavisualization}
\usetikzlibrary{datavisualization.formats.functions}

\usetikzlibrary{arrows}

\usepackage{import}

\usepackage{mathtools}
\usepackage{mathrsfs}
\usepackage{commath}

\usepackage{framed}

\usetikzlibrary{calc}

\makeatletter
\pgfarrowsdeclare{X}{X}
{
  \pgfutil@tempdima=0.3pt%
  \advance\pgfutil@tempdima by.25\pgflinewidth%
  \pgfutil@tempdimb=5.5\pgfutil@tempdima\advance\pgfutil@tempdimb by.5\pgflinewidth%
  \pgfarrowsleftextend{+-\pgfutil@tempdimb}
  \pgfarrowsrightextend{0pt}
}
{
  \pgfutil@tempdima=0.3pt%
  \advance\pgfutil@tempdima by.25\pgflinewidth%
  \pgfsetdash{}{+0pt}
  \pgfsetroundcap
  \pgfsetmiterjoin
  \pgfpathmoveto{\pgfqpoint{-5.5\pgfutil@tempdima}{-6\pgfutil@tempdima}}
  \pgfpathlineto{\pgfqpoint{5.5\pgfutil@tempdima}{6\pgfutil@tempdima}}
  \pgfpathmoveto{\pgfqpoint{-5.5\pgfutil@tempdima}{6\pgfutil@tempdima}}
  \pgfpathlineto{\pgfqpoint{5.5\pgfutil@tempdima}{-6\pgfutil@tempdima}}
  \pgfusepathqstroke 
}
\makeatother

\newtheorem{theorem}{Theorem}[section]
\newtheorem{lemma}[theorem]{Lemma}
\newtheorem{proposition}[theorem]{Proposition}
\newtheorem{corollary}[theorem]{Corollary}
\newtheorem{example}[theorem]{Example}
\theoremstyle{definition}
\newtheorem{definition}[theorem]{Definition}

\theoremstyle{remark}
\newtheorem{remark}[theorem]{Remark}

\usepackage{relsize}
\usepackage{exscale}

\usepackage{mathtools}
\usepackage{mathrsfs}

\usepackage{framed}

\usepackage{datetime}

\usepackage[pagewise]{lineno}
\usepackage{hyperref} 

\renewcommand{\epsilon}{\varepsilon}

\DeclareMathOperator{\Orb}{Orb}

\title[Shadowing, ICT and $\alpha$-limit sets]{Shadowing, Internal Chain Transitivity and $\mathlarger{\mathlarger{{\mathlarger{\alpha}}}}$-limit sets} 
\author[Good, Meddaugh and Mitchell]{Chris Good, Jonathan Meddaugh and Joel Mitchell}

\begin{document}

\hypersetup{pageanchor=false} 
\maketitle



\begin{abstract}
Let $f \colon X \to X$ be a continuous map on a compact metric space $X$ and let $\alpha_f$, $\omega_f$ and $ICT_f$ denote the set of $\alpha$-limit sets, $\omega$-limit sets and nonempty closed internally chain transitive sets respectively. We show that if the map $f$ has shadowing then every element of $ICT_f$ can be approximated (to any prescribed accuracy) by both the $\alpha$-limit set and the $\omega$-limit set of a full-trajectory. Furthermore, if $f$ is additionally expansive then every element of $ICT_f$ is equal to both the $\alpha$-limit set and the $\omega$-limit set of a full-trajectory. In particular this means that shadowing guarantees that $\overline{\alpha_f}=\overline{\omega_f}=ICT_f$ (where the closures are taken with respect to the Hausdorff topology on the space of compact sets), whilst the addition of expansivity entails $\alpha_f=\omega_f=ICT_f$. We progress by introducing novel variants of shadowing which we use to characterise both maps for which $\overline{\alpha_f}=ICT_f$ and maps for which $\alpha_f=ICT_f$. 
\end{abstract}



\hypersetup{pageanchor=true} 






\section{Introduction}

Let $f \colon X \to X$ be a dynamical system, so that $f$ is a continuous map on the compact metric space $X$. Given a point $x \in X$, its \emph{$\omega$-limit set} is the set of accumulation points of the sequence $x, f(x), f^2(x), \ldots$. Calculating the $\omega$-limit set of a given point is often relatively easy. Conversely one may ask if a given set is an $\omega$-limit set: this can be quite difficult to answer. As such, various authors have either studied, or attempted to characterise, the set of all $\omega$-limit sets, denoted here by $\omega_f$, in a variety of settings. For example, $\omega$-limit sets of continuous maps of the closed unit interval $I$ have been completely characterised in \cite{AgronskyBrucknerCederPearson,BrucknerSmital}: the authors show that a nonempty subset $E$ of $I$ is an $\omega$-limit set of some continuous map $f$ if and only if $E$ is either a closed, nowhere dense set, or a union of finitely many non-degenerate closed intervals. Furthermore, it has been shown that $\omega_f$ is closed (with respect to the Hausdorff topology) for maps of the circle \cite{Pokluda}, the interval \cite{BlokhBrucknerHumkeSmital} and other finite graphs \cite{MaiShao}. It is known \cite{Hirsch} that every $\omega$-limit set is \emph{internally chain transitive}: briefly a set $A \subseteq X$ is internally chain transitive if for any $a,b \in A$ and any $\epsilon>0$ there exists a finite sequence $\langle x_0, x_1, \ldots, x_n \rangle$ in $A$ such that $x_0=a$, $x_n=b$ and $d(f(x_i), x_{i+1})< \epsilon$ for each $i$. We denote the set of nonempty closed internally chain transitive sets by $ICT_f$. The map $f$ is said to have \emph{shadowing} if for each $\epsilon>0$ there is a $\delta>0$ such that for any sequence $\langle x_i \rangle _{i =0} ^\infty$ with $d(f(x_i), x_{i+1}) < \delta$ for each $i$, there is a point $z \in X$ such that $d(f^i(z) , x_i)< \epsilon$ for each $i$. In this case we say $z$ \emph{shadows} or $\epsilon$\emph{-shadows} the sequence $\langle x_i \rangle _{i =0} ^\infty$. Shadowing has both numerical and theoretical importance and has been studied extensively in a variety of settings; in the context of Axiom A diffeomorphisms \cite{bowen-markov-partitions}, in numerical analysis \cite{Corless,CorlessPilyugin,Pearson}, as an important factor in stability theory \cite{Pilyugin, robinson-stability,Walters}, in understanding the structure of $\omega$-limit sets and Julia sets \cite{BarwellGoodOprochaRaines, BarwellMeddaughRaines2015, BarwellRaines2015, Bowen, MeddaughRaines}, and as a property in and of itself \cite{Coven1, GoodMeddaugh2018, GoodMitchellThomas, LeeSakai, Nusse, Pennings, Pilyugin,Sakai2003}. A variety of variants of shadowing have also been studied including, for example, ergodic, thick and Ramsey shadowing \cite{brian-oprocha, bmr, Dastjerdi, Fakhari, Oprocha2016}, limit, or asymptotic, shadowing \cite{BarwellGoodOprocha, GoodOprochaPuljiz2019, Pilyugin2007}, $s$-limit shadowing \cite{BarwellGoodOprocha,GoodOprochaPuljiz2019, LeeSakai}, orbital shadowing \cite{GoodMeddaugh2016,Mitchell, PiluginRodSakai2002,  Pilyugin2007}, and inverse shadowing \cite{CorlessPilyugin, GoodMitchellThomas2, Lee}.

Of particular importance to us is a result of Meddaugh and Raines \cite{MeddaughRaines} who establish that, for maps with shadowing, $\overline{\omega_f}=ICT_f$. More recently, using novel variants of shadowing, Good and Meddaugh \cite{GoodMeddaugh2016} precisely characterised maps for which $\overline{\omega_f}=ICT_f$ and $\omega_f=ICT_f$. 

Whilst the $\omega$-limit set of a point can be thought of as its \emph{target} - it is where the point \emph{ends up} - an $\alpha$-limit set concerns where a point came from - its source, so to speak. 
However, whilst the definition of an $\omega$-limit set is fairly natural, giving an appropriate definition of an $\alpha$-limit set is less straightforward. This is because a point may have multiple points in its preimage (or indeed, if the map is not surjective, it may have empty preimage). Various approaches to this difficulty have been taken; these will be discussed in more detail in Section \ref{SectionAlphaLimitTypes}. We follow the approach taken in \cite{BalibreaPiotr} and \cite{Hirsch}, by refraining from defining such sets for individual points, but rather defining them for \emph{backward trajectories}. Given a point $x \in X$ an infinite sequence $\langle x_i \rangle _{i \leq 0}$ is called a \emph{backward trajectory} of $x$ if $f(x_i)=x_{i+1}$ for all $i \leq -1$ and $x_0=x$. The $\alpha$-limit set of $\langle x_i \rangle _{i \leq 0}$ is the set of accumulation points of this sequence. We denote the set of all $\alpha$-limit sets by $\alpha_f$. Although $\alpha$-limit sets have not been studied quite as extensively as there $\omega$ counterparts, interest in them has been growing (see, for example, \cite{BalibreaPiotr, Coven,CuiDing, Hero, Hirsch}). 

As with $\omega$-limit sets, it is known that $\alpha$-limit sets are internally chain transitive \cite{Hirsch}. In this paper we seek to provide a characterisation of maps for which $\alpha_f$ and $ICT_f$ coincide. We start with the preliminaries in Section \ref{SectionPrelim}. Section \ref{SectionAlphaLimitTypes} is a standalone section in which we briefly explain the various types of $\alpha$-limit sets that have been studied in the literature. 
In Section \ref{SectionShadowing} we show that, for maps with shadowing, for any $\epsilon>0$ and any $A\in ICT_f$ there is a full trajectory whose $\alpha$-limit set and $\omega$-limit set both lie within $\epsilon$ of $A$ (with respect to the Hausdorff distance). Furthermore, we show that the addition of expansivity entails that there is a full trajectory whose limit sets equal $A$. In particular this means that for maps with shadowing $\overline{\alpha_f}=\overline{\omega_f}=ICT_f$, whilst the addition of expansivity means that $\alpha_f=\omega_f=ICT_f$. We progress in Section \ref{SectionCharacterise} by introducing novel types of shadowing which we use to characterise both maps for which $\overline{\alpha_f}=ICT_f$ and maps for which $\alpha_f=ICT_f$, complementing the work of the first and second author in \cite{GoodMeddaugh2016}.





\section{Preliminaries}\label{SectionPrelim}
A \emph{dynamical system} is a pair $(X,f)$ consisting of a compact metric space $X$ and a continuous function $f\colon X \to X$. We say the \emph{positive orbit of $x$ under $f$} is the set of points $\{x, f(x), f^2(x), \ldots\}$; we denote this set by $\Orb^+ _f(x)$. A \emph{backward trajectory} of the point $x$ is a sequence $\langle x_{i}\rangle_{i\leq0}$ for which $f(x_{i})=x_{i+1}$ for all $i\leq -1$ and $x_0=x$. We say a bi-infinite sequence $\langle x_i \rangle _{i \in \mathbb{Z}}$ is a \emph{full orbit} (of each $x_i$) if $f(x_i)=x_{i+1}$ for each $i \in \mathbb{Z}$. We emphasise that a full orbit of a point need not be unique. Note further that we do not assume that the map $f$ is a surjection. (NB. Because we will be particularly concerned with backward accumulation points of individual trajectories, for clarity we will say that a point which does not have an infinite backward trajectory does not have a full orbit. Whenever we say \emph{full orbit}, we mean a bi-infinite trajectory.)


For a sequence $\langle x_i \rangle _{i>N}$ in $X$, where $N \geq -\infty$, we define its $\omega$-limit set, denoted $\omega(\langle x_i\rangle_{i > N} )$, or simply $\omega(\langle x_i\rangle)$, to be the set of accumulation points of the positive tail of the sequence. Formally: 
\[\omega(\langle x_i\rangle)= \bigcap_{M \in \mathbb{N}}\overline{\{x_n \mid n >M\}}.\]
For $x \in X$, we define the $\omega$\emph{-limit set of }$x$: $\omega(x) \coloneqq \omega(\langle f^n(x)\rangle_{n = 0} ^\infty)$. 
In similar fashion, for a sequence $\langle x_i \rangle _{i< N}$ in $X$, where $N \leq \infty$, we define its \emph{$\alpha$-limit set}, denoted $\alpha(\langle x_i\rangle_{i < N} )$, or simply $\alpha(\langle x_i\rangle)$, to be the set of accumulation points of the negative tail of the sequence. Formally: 
\[\alpha(\langle x_i \rangle )= \bigcap_{M \in \mathbb{N}}\overline{\{x_n \mid n <-M\}}.\]
We denote by $\omega_f$ the set of all $\omega$-limit sets of points in $X$. We denote by $\alpha_f$ the set of all $\alpha$-limit sets of full trajectories in $(X,f)$. Note that since $X$ is compact it follows that elements of $\alpha_f$ and $\omega_f$ are closed, compact and nonempty.

We denote by $2^X$ the hyperspace of nonempty compact subsets of $X$. This is a (compact) metric space in its own right with the \emph{Hausdorff metric} induced by the metric $d$. For $A,B \in 2^X$ the \emph{Hausdorff distance} between $A$ and $B$ is given by
\[d_H (A,A^\prime)= \inf \{\epsilon>0 \mid A \subseteq B_\epsilon (A^\prime) \text{ and } A^\prime \subseteq B_\epsilon (A)\}. \]
Note that, as collections of nonempty compact sets, $\alpha_f$ and $\omega_f$ are subsets of $2^X$.

A set $A \subseteq X$ is said to be \emph{invariant} if $f(A) \subseteq A$. It is \emph{strongly invariant} if $f(A)=A$. A nonempty closed set $A$ is \emph{minimal} if $\omega(x)=A$ for all $x \in A$.

A finite or infinite sequence $\langle x_i\rangle_{i=0} ^N$ is said to be an \emph{$\epsilon$-chain} if $d(f(x_i), x_{i+1})< \epsilon$ for all indices $i<N$. If $N=\infty$ then we say the sequence is an \emph{$\epsilon$-pseudo-orbit.} A set $A$ is \emph{internally chain transitive} if for any pair of points $a,b \in A$ and any $\epsilon>0$ there exists a finite $\epsilon$-chain $\langle x_i \rangle _{i=0} ^N$ in $A$ with $x_0=a$, $x_N=b$ and $N\geq 1$. We denote by $ICT_f$ the set of all nonempty closed internally chain transitive sets. Notice that $ICT_f \subseteq 2^X$. Meddaugh and Raines \cite{MeddaughRaines} establish the following result.
\begin{lemma}\textup{\cite{MeddaughRaines}}\label{lemmaICTclosed}
Let $(X,f)$ be a dynamical system. Then $ICT_f$ is closed in $2^X$.
\end{lemma}
Hirsch \emph{et al.}\ \cite{Hirsch} show that the $\alpha$-limit set (resp.\ $\omega$-limit set) of any pre-compact backward (resp.\ forward) trajectory is internally chain transitive. Since our setting is a compact metric space all $\alpha$- and $\omega$- limit sets are internally chain transitive. We formulate this as Lemma \ref{LemmaAlphaOmegaAreICT} below.
\begin{lemma}\textup{\cite{Hirsch}} \label{LemmaAlphaOmegaAreICT} Let $(X,f)$ be a dynamical system. Then $\alpha_f \subseteq ICT_f$ and $\omega_f \subseteq ICT_f$.
\end{lemma}

\begin{remark}\label{RemarkOmegaAlphaDistinct}
When one first encounters positive and negative limit sets of trajectories, it is natural to ask (for a surjective map) if every $\omega$-limit set is also an $\alpha$-limit set, along with the converse. The following is an example of a homeomorphism for which neither is true. Take two copies of the interval and embed them side by side in the plane (i.e.\ one on the left and one on the right). Snake one infinite line between them which has each interval as an accumulation set - akin to how the topologist's sine curve approaches the $y$-axis. Define a continuous map as follows: Let every point on each of the two intervals be fixed whilst points on the line move continuously along it, away from the left interval and towards the right. 
It follows that the left interval is the $\alpha$-limit set of the unique backward trajectory of any point on the line, whilst the right left interval is the $\omega$-limit set of any point on the line. However it is clear that the left interval is not an $\omega$-limit set, whilst the right interval is not an $\alpha$-limit set.
\end{remark}

\begin{remark}
As stated in \cite{BalibreaPiotr}, a minimal set is both an $\omega$-limit set and an $\alpha$-limit set.
\end{remark}

Whilst it may be the case that $\alpha_f \neq \omega_f$, 
it is true that every $\alpha$-limit set contains the $\omega$-limit set of every one of its points and, similarly, every $\omega$-limit set contains an $\alpha$-limit set of a backward trajectory of each of its points. To show this we recall the well-known fact that the $\omega$-limit sets in compact systems are strongly invariant (e.g.\ \cite[Theorem 3.1.9]{deVries}). The same is true of the $\alpha$-limit sets of backward trajectories (e.g.\ \cite[Lemma 1]{BalibreaPiotr}).

\begin{proposition}\label{propositionContainment}
Let $x,y \in X$ and suppose that $\langle z_{i}\rangle_{i\leq0}$ is a backward trajectory of a point $z=z_0 \in X$. Then:
\begin{enumerate}
    \item If $x \in \alpha (\langle z_{i}\rangle)$ then $\overline{\Orb_f ^+(x)} \subseteq \alpha (\langle z_{i}\rangle) $.
    \item If $y \in \omega(x)$ then there is a backward trajectory $\langle y_{i}\rangle_{i\leq0}$, with $y_0=y$, which lies in $\omega(x)$ and such that $\alpha (\langle y_{i}\rangle)  \subseteq \omega(x)$.
\end{enumerate}
\end{proposition}

\begin{proof}
Condition (1) is immediate from the fact that $\alpha$-limit sets are closed and invariant under $f$.

Now suppose $y \in \omega(x)$ and let $y_0=y$. Since $\omega$-limit sets are strongly invariant $y$ has a preimage in $\omega(x)$, call it $y_{-1}$. This itself has a preimage in $\omega(x)$; call it $y_{-2}$. Continuing in this manner gives a backward trajectory $\langle y_{i}\rangle_{i\leq0}$ of $y$ which lies in $\omega(x)$. The result now follows by observing that $\omega(x)$ is closed.
\end{proof}

\begin{remark}
In \cite{Hero} the author proves condition (1) in Proposition \ref{propositionContainment} holds for interval maps.
\end{remark}

A point $x$ is said to \emph{$\epsilon$-shadow} a sequence $\langle x_i\rangle_{i=0} ^\infty$ if $d(f^i(x), x_i)< \epsilon$ for all $i \in \mathbb{N}_0$. We say the system $(X,f)$ has the \emph{shadowing property}, or simply shadowing, if for every $\epsilon>0$ there exists $\delta>0$ such that every $\delta$-pseudo-orbit is $\epsilon$-shadowed. 

\begin{definition}
Suppose that $(X,f)$ is a dynamical system.
\begin{enumerate}
\item The sequence $\langle x_{i}\rangle_{i\leq0}$ is a \emph{backward $\delta$-pseudo-orbit} if $d(f(x_{i}),x_{i+1})<\delta$ for each $i\leq-1$.
\item The sequence $\langle x_{i}\rangle_{i\in\mathbb Z}$ is a \emph{two-sided $\delta$-pseudo-orbit} if $d(f(x_{i}),x_{i+1})<\delta$ for each $i\in\mathbb Z$.
\item The system $(X,f)$ has \emph{backward shadowing} if for any $\epsilon>0$ there exists $\delta>0$ such that for any backward $\delta$-pseudo-orbit $\langle x_i \rangle _{i \leq 0}$ there exists a backward trajectory $\langle z_i \rangle _{i \leq 0}$ such that $d(x_i, z_i) < \epsilon$ for all $i \leq 0$.
\item The system $(X,f)$ has \emph{two-sided shadowing} if for any $\epsilon>0$ there exists $\delta>0$ such that for any two-sided $\delta$-pseudo-orbit $\langle x_i \rangle _{i \in \mathbb{Z}}$ there exists a full trajectory $\langle z_i \rangle _{i \in \mathbb{Z}}$ such that $d(x_i, z_i) < \epsilon$ for all $i \in \mathbb{Z}$.
\end{enumerate}
\end{definition}
A sequence $\langle x_{i}\rangle_{i=0} ^\infty$ is called an \emph{asymptotic pseudo-orbit} if $d(f(x_i), x_{i+1}) \rightarrow 0$ as $i \rightarrow \infty$. Similarly a sequence $\langle x_{i}\rangle_{i\leq0}$ is a \emph{backward asymptotic pseudo-orbit} if $d(f(x_{i}),x_{i+1})\rightarrow 0$ as $i \rightarrow -\infty$. Finally a sequence $\langle x_{i}\rangle_{i \in \mathbb{Z}}$ is called a \emph{two-sided asymptotic pseudo-orbit} if $d(f(x_i), x_{i+1}) \rightarrow 0$ as $i \rightarrow \pm \infty$.

The system $(X,f)$ has \emph{s-limit shadowing} if, in addition to having shadowing, for any $\epsilon>0$ there exists $\delta>0$ such that for any asymptotic $\delta$-pseudo orbit $\langle x_i \rangle_{i=0} ^\infty$ there exists $z \in X$ which \emph{asymptotically} $\epsilon$-shadows $\langle x_i \rangle_{i=0} ^\infty$ (i.e.\ $d(f^i(z), x_i) \to 0$ as $i \to \infty$ and $d(f^i(z), x_i) < \epsilon$ for all $i \in \mathbb{N}_0$). The system has \emph{two-sided s-limit shadowing} if, in addition to two-sided shadowing, for any $\epsilon>0$ there exists $\delta>0$ such that for any two-sided asymptotic $\delta$-pseudo orbit $\langle x_i \rangle_{i\in \mathbb{Z}}$ there exists a full trajectory $\langle z_i \rangle _{i \in \mathbb{Z}}$ which asymptotically $\epsilon$-shadows $\langle x_i \rangle_{i \in \mathbb{Z}}$ (i.e.\ $d(f^i(z), x_i) \to 0$ as $i \to \pm\infty$ and $d(f^i(z), x_i)<\epsilon$ for all $i \in \mathbb{Z}$).

\subsection{Shift spaces}\label{subsectionShiftSpaces}
Given a finite set $\Sigma$ considered with the discrete topology, \emph{the one-sided full shift with alphabet }$\Sigma$ consists of the set of infinite sequences in $\Sigma$, that is $\Sigma^{\mathbb{N}_0}$, which we consider with the product topology. This forms a dynamical system with the \emph{shift map} $\sigma$, given by \[\sigma\left(\langle a_i \rangle _{i\geq 0}\right)=\langle a_{i+1} \rangle _{i \geq 0}. \]
A \emph{one-sided shift space} is some compact strongly invariant (under $\sigma$) subset of some one-sided full shift. 

In similar fashion, the \emph{two-sided full shift with alphabet} $\Sigma$ consists of the set of bi-infinite sequences in $\Sigma$, that is $\Sigma^{\mathbb{Z}}$, which we consider with the product topology. As before, this forms a dynamical system with the \emph{shift map} $\sigma$, which we define by saying that, for each $i \in \mathbb{Z}$, \[\pi_i(\sigma\left(\langle a_i \rangle _{i\in \mathbb{Z}}\right))= a_{i+1}, \]
where $\pi_i$ is the projection map for each $i$.
A \emph{two-sided shift space} is some compact strongly invariant (under $\sigma$) subset of some two-sided full shift.  If $(X, \sigma)$ is a two-sided shift space and $x= \langle a_i \rangle _{i \in \mathbb{Z}} \in X$ then we refer to the sequences $ \langle a_i \rangle _{i \geq 0}$ and  $ \langle a_i \rangle _{i \leq 0}$ as the \emph{right-tail} and \emph{left-tail} of $x$ respectively.

Given an alphabet $\Sigma$, a word in $\Sigma$ is a finite sequence $a_0 a_1\ldots a_m$, made up of elements of $\Sigma$. Let $\mathcal{F}$ be a finite set of words in $\Sigma$. The \emph{one-sided shift of finite type associated with $\mathcal{F}$} is the dynamical system $(X_\mathcal{F}, \sigma)$ where $X_\mathcal{F}$ is the set of all infinite sequences which do not contain any occurrence of any word from $\mathcal{F}$. The \emph{two-sided shift of finite type associated with $\mathcal{F}$} is the dynamical system $(Z_\mathcal{F}, \sigma)$ where $Z_\mathcal{F}$ is the set of all bi-infinite sequences which do not contain any occurrence of any word from $\mathcal{F}$. A shift space $(X,\sigma)$ is said to be a \emph{one-sided (resp.\ two-sided) shift of finite type} if there exists a finite set of words $\mathcal{F}$ such that $X= X_\mathcal{F}$ (resp.\ $X=Z_\mathcal{F}$).

If $(X, \sigma)$ is a one-sided shift space, $x= \langle a_i \rangle _{i \geq 0} \in X$ and $n \in \mathbb{N}_0$, we refer to the word $a_0a_1\ldots a_n$ as an \emph{initial segment} of $x$. In similar fashion, if $(X, \sigma)$ is a two-sided shift space and $x= \langle a_i \rangle _{i \in \mathbb{Z}} \in X$ and $n \in \mathbb{N}_0$, we refer to the word $a_{-n}\ldots a_{-1} a_0a_1\ldots a_n$ as a \emph{central segment} of $x$. In the two-sided case, when writing out an element of $X$ in full we use a ``$\cdot$" to indicate the position of the middle of the central segment:
\[ x=\ldots a_{-3} a_{-2} a_{-1} \cdot  a_0 a_1 a_2 a_3 \ldots.\]

The following two theorems concerning limit sets in shift spaces are folklore.

\begin{theorem}\label{thmLimitSetsInOneSidedShiftSpace}
Let $(X,\sigma)$ be a one-sided shift space. Let $x,y \in X$. Then $y \in \omega(x)$ if and only if every initial segment of $y$ occurs infinitely often in $x$. Given a backward trajectory $\langle x_i \rangle _{i \leq 0}$ consider the backward infinite sequence $\langle a_i \rangle _{i \leq 0}$ where $a_i = \pi_0 (x_i)$. Then $y \in \alpha(\langle x_i \rangle )$ if and only if every initial segment of $y$ occurs infinitely often in $\langle a_i \rangle _{i \leq 0}$.
\end{theorem}

\begin{theorem}\label{thmLimitSetsInTwoSidedShiftSpace}
Let $(X,\sigma)$ be a two-sided shift space. Let $x,y \in X$. Then $y \in \omega(x)$ if and only if every central segment of $y$ occurs infinitely often in the right-tail of $x$. Given a backward trajectory $\langle x_i \rangle _{i \leq 0}$ then $y \in \alpha(\langle x_i \rangle )$ if and only if every central segment of $y$ occurs infinitely often in the left-tail of $x_0$.
\end{theorem}

For those wanting more information about shift systems, \cite[Chapter 5]{deVries} provides a thorough introduction to the topic.

As stated in Lemma \ref{LemmaAlphaOmegaAreICT}, $\alpha_f$ and $\omega_f$ are both subsets of $ICT_f$. Example \ref{ExampleShiftSpaceAlphaOmegaDifferent} gives a surjective shift space $(X, \sigma)$ where $\alpha_\sigma$, $\omega_\sigma$ and $ICT_\sigma$ are all distinct, complementing the discussion in Remark \ref{RemarkOmegaAlphaDistinct}.

\begin{example}\label{ExampleShiftSpaceAlphaOmegaDifferent} Let
$x=1010^210^3\ldots$,
and 
$y=2020^220^3\ldots$. Let \[P(x)=\{ 30^n30^{n-1}\ldots30x \mid n \in \mathbb{N}\}.\]
Take
\[ X = \overline{\bigcup_{z \in P(x)} \Orb^+ _\sigma(z) \cup \Orb^+ _\sigma(y)\cup \{ 0^ny \mid n \in \mathbb{N}\}} ,\]
where the closure is taken with regard to the one-sided full shift on the alphabet $\{0,1,2,3\}$. Considering the system $(X,\sigma)$, $\alpha_\sigma \neq \omega_\sigma \neq ICT_\sigma$. Furthermore $\alpha_\sigma \not\subseteq \omega_\sigma$ and $\omega_\sigma \not\subseteq \alpha_\sigma$.
\end{example}
In Example \ref{ExampleShiftSpaceAlphaOmegaDifferent}, $\omega(x)=\{0^\infty, 0^n10^\infty \mid n \geq 0\}$ and $\omega(y)=\{0^\infty, 0^n20^\infty \mid n \geq 0\}$. It is easy to see that the only other $\omega$-limit set is $\{0^\infty\}$. Thus \[\omega_\sigma=\{\{0^\infty\}, \{0^\infty, 0^n10^\infty \mid n \geq 0\}, \{0^\infty, 0^n20^\infty \mid n \geq 0\}\}.\] Meanwhile \[\alpha_\sigma=\{\{0^\infty\}, \{ 0^\infty, 0^n30^\infty \mid n \geq 0\}\}.\] Finally whilst $ICT_\sigma\supseteq \alpha_\sigma \cup \omega_\sigma$ it additionally contains $\{0^\infty, 0^n10^\infty, 0^n20^\infty \mid n \geq 0\}$, $\{0^\infty, 0^n10^\infty, 0^n30^\infty \mid n \geq 0\} \}$, $\{0^\infty, 0^n20^\infty, 0^n30^\infty \mid n \geq 0\}$ and\newline $\{0^\infty, 0^n10^\infty, 0^n20^\infty, 0^n30^\infty \mid n \geq 0\}$. Hence $\alpha_\sigma \neq \omega_\sigma \neq ICT_\sigma$, $\alpha_\sigma \not\subseteq \omega_\sigma$ and $\omega_\sigma \not\subseteq \alpha_\sigma$.
\section{Various notions of negative limit sets}\label{SectionAlphaLimitTypes}
In the previous section we defined what we mean by the term $\alpha$-limit set: it was defined for backward sequences. Meanwhile the definition of an $\omega$-limit set was extended to individual points. This was done in the only natural way: any given point only has one forward orbit. If one wishes to define the $\alpha$-limit set of a point, say $x$, the best way forward is less obvious; there are multiple approaches one might reasonably take when defining negative limit sets of points. In this standalone section we give a brief outline of several different approaches taken in the literature and give two examples which serve to illustrate their differences. 

For homeomorphisms one can define $\alpha$-limit sets (or negative limit sets) in precisely the same way as $\omega$-limit sets. With non-invertible maps, however, a seemingly natural definition is less obvious. One approach is to take the set of accumulation points of the sequence of sets $f^{-k}(\{x\})$: this is done in \cite{Coven} and \cite{CuiDing}. Call this Approach 1 (A1). Two further approaches are motivated by considering the accumulation points of backward trajectories of the point in question. 
One might say that $y$ is in the negative limit set of a point $x$ if there exists a sequence $\langle y_i\rangle_{i=0} ^\infty$ such that $y_i \in \Orb^+ _f(y_{i+1})$ for each $i$, $x=y_0$ and $\lim_{i \to \infty} y_i = y$: that is, the negative limit set of $x$ is the union of all accumulation points of backward trajectories from $x$. In \cite{Hero} the author defines this set as the \emph{special $\alpha$-limit set of $x$} and examines them for interval maps. These sets are investigated in \cite{Sun} and \cite{Sun2} for graph maps and dendrites. Call this Approach 2 (A2). The final approach, A3, used in \cite{Hero}, is to say $y$ is in the $\alpha$-limit set of a point $x$ if there exists a sequence $\langle y_i\rangle_{i=1} ^\infty$ and a strictly increasing sequence $\langle n_i\rangle_{i=1} ^\infty$ such that $f^{n_i}(y_i)=x$ for each $i$ and $\lim_{i \to \infty} y_i = y$. Clearly this set contains the one given by A2. The converse is not true (see Example \ref{ExampleA2vsA3}). 


By means of demonstrating some of the differences A1-3 yield we provide the following two examples. 

\begin{example}\label{ExampleA1vsA2and3}
Define a map $f\colon [-1,1] \to [-1,1]$ by
\[f(x)=\left\{\begin{array}{lll}
2x+1 & \text{if} & x \in [-1,-\nicefrac{1}{2}),
\\0 & \text{if} & x \in [-\nicefrac{1}{2},\nicefrac{1}{2}) ,
\\ 2x-1 & \text{if} & x \in [\nicefrac{1}{2},1].
\end{array}\right.\]
\end{example}

\begin{figure}[h]
\centering
\begin{tikzpicture}[scale=2.5]
\datavisualization [school book axes,
                    visualize as smooth line,
                    y axis={label},
                    x axis={label} ]

data [format=function] {
      var x : interval [-1:-0.5] samples 100;
      func y = 1+2* \value x;
      }
data [format=function] {
      var x : interval [-0.5:0.5] samples 100;
      func y = 0;
      }
data [format=function] {
      var x : interval [0.5:1] samples 100;
      func y = -1+2 * \value x;
      };
\draw (-0.07,1) -- (0.07,1);      
\node at (-0.15,1) {\footnotesize 1};
\end{tikzpicture}
\caption{Example \ref{ExampleA1vsA2and3}}
\label{figureES}
\end{figure}
In Example \ref{ExampleA1vsA2and3}, under A1 the negative limit set of $0$ can be seen to be the whole interval $[-1,1]$. Under A2 and A3 the negative limit set of $0$ is simply $\{-1,0,1\}$. Notice that the negative limit set of any backward trajectory from $0$ will be either $\{-1\}$ or $\{0\}$ or $\{1\}$.

\begin{example}\label{ExampleA2vsA3}
Define a map $f\colon [-1,2] \to [-1,2]$ by
\[f(x)=\left\{\begin{array}{lll}
2x+2 & \text{if} & x \in [-1,0),
\\2-2x & \text{if} & x \in [0,1) ,
\\ 2x-2 & \text{if} & x \in [1,2].
\end{array}\right.\]
\end{example}

\begin{figure}[h]
\centering
\begin{tikzpicture}[scale=2]
\datavisualization [school book axes,
                    visualize as smooth line,
                    y axis={label},
                    x axis={label} ]

data [format=function] {
      var x : interval [-1:0] samples 100;
      func y = 2+2* \value x;
      }
data [format=function] {
      var x : interval [0:1] samples 100;
      func y = 2+(-2)*\value x;
      }
data [format=function] {
      var x : interval [1:2] samples 100;
      func y = -2+2 * \value x;
      };
\end{tikzpicture}
\caption{Example \ref{ExampleA2vsA3}}
\label{figureES}
\end{figure}
In Example \ref{ExampleA2vsA3}, under A2 the negative limit set of $0$ is $\{2/3, 2\}$. Consider the backward trajectory of $0$ given by the increasing sequence $\langle x_i \rangle_{i \geq 0} $, where $x_0=0$ and $x_1=1$, $x_2=\frac{3}{2}$, $x_3= \frac{7}{4}$.... This sequence approaches $2$. However each point $x_i$ in this sequence has a preimage $y_i$ in the interval $[-1,0)$. Each of these $y_i$ thereby eventually map onto $0$ but they do not themselves have preimages. Furthermore, if $f^n(y_i)=0$ and $f^m(y_{i+1})=0$ then by construction $m>n$. This, together with the fact that $\lim_{i \to \infty} y_i =0$ implies that $0$ is in the negative limit set of itself under A3. Under A3 the negative limit set of $0$ is $\{0, 2/3, 2\}$. (NB. Hero \cite{Hero} provides an example illustating this same difference. For Hero, $0$ would be an $\alpha$-limit point of itself but not a special $\alpha$-limit point of itself: these would only be $2/3$ and $2$.)

\medskip

As stated previously, in this paper we will not define $\alpha$-limit sets of individual points, instead we focus on the accumulation points of individual backward trajectories. Note that this is the approach taken in \cite{BalibreaPiotr} and \cite{Hirsch}.

\section{Shadowing, ICT and $\alpha_f$}\label{SectionShadowing}
The following lemma is a recent observation of the authors \emph{et al.}\ (see \cite{GoodMaciasMeddaughMitchellThomas}). 

\begin{lemma}\label{lemmaBackwardsShadowing}\textup{\cite{GoodMaciasMeddaughMitchellThomas}}
Let $(X,f)$ be a dynamical system with $X$ compact. If $f$ has shadowing then it has backward shadowing and two-sided shadowing. If $f$ is onto then all three properties are equivalent. 
\end{lemma}

\begin{theorem}\label{thmShadICT} Let $(X,f)$ be a dynamical system with shadowing. Then for any $\epsilon>0$ and any $A \in ICT_f$ there is a full trajectory $\langle x_i \rangle_{i \in \mathbb{Z}}$ such that
\begin{enumerate}
    \item $d_H(\omega(x_0), A) < \epsilon$
    \item $d_H(\alpha(\langle x_i \rangle), A )< \epsilon$.
\end{enumerate}
In particular every element of $ICT_f$ is either in, or is a limit point of, both $\alpha_f$ and $\omega_f$.
\end{theorem}

\begin{remark}\label{remarkEtaDense}Before proving Theorem \ref{thmShadICT}, we observe that for any $A \in ICT_f$, for each $\eta>0$ and for each $a \in A$ there exists a finite $\eta$-chain $\langle a=a_0, a_1, \ldots , a_m=a \rangle$ in $A$ which is \emph{$\eta$-dense in $A$}, i.e.\ for each $i \in \{0, \ldots, m\}$, $a_i \in A$ and 
$\bigcup_{i=0} ^m B_\eta (a_i) \supseteq A.$
\end{remark}

\begin{proof}[Proof of Theorem \ref{thmShadICT}]
Let $A \in ICT_f$ and let $\epsilon>0$ be given. By Lemma \ref{lemmaBackwardsShadowing} there exists $\delta>0$ such that every two-sided $\delta$-pseudo-orbit is $\nicefrac{\epsilon}{2}$-shadowed by a full orbit. We will construct a two-sided asymptotic $\delta$-pseudo-orbit in $A$ which is $\eta$-dense in $A$ for all $\eta>0$. To this end, let $l \in \mathbb{N}$ be such that $\nicefrac{1}{2^{l}}< \delta$. Pick $b \in A$. For each $k \in \mathbb{N}_0$ choose a finite $\nicefrac{1}{2^{l+k}}$-chain $\langle a_{k \cdot 0}=b, a_{k \cdot 1}, a_{k \cdot 2}, \ldots, a_{k \cdot m_k} \rangle$ in $A$ which is $\nicefrac{1}{2^{l+k}}$-dense in $A$ and such that $d(f( a_{k \cdot m_k}), b) < \nicefrac{1}{2^{l+k}}$. (Here we are simply using the observation in Remark \ref{remarkEtaDense}.) Concatenation of these chains now gives us an asymptotic $\delta$-pseudo-orbit in $A$:
\[ \langle a_{0 \cdot 0}, a_{0 \cdot 1}, a_{0 \cdot 2}, \ldots, a_{0 \cdot m_0}, a_{1 \cdot 0}, a_{1 \cdot 1}, a_{1 \cdot 2}, \ldots, a_{1 \cdot m_1}, \ldots, a_{k \cdot 0}, a_{k \cdot 1}, a_{k \cdot 2}, \ldots, a_{k \cdot m_k}, \ldots \rangle. \]
We can now extend this into a two-sided asymptotic $\delta$-pseudo-orbit in $A$ by `running backwards' through the $\delta$-chains:

\[\langle  \ldots, a_{2 \cdot 0}, a_{2 \cdot 1}, \ldots, a_{2 \cdot m_2}, a_{1 \cdot 0}, a_{1 \cdot 1}, \ldots, a_{1 \cdot m_1} \cdot  a_{0 \cdot 0}, a_{0 \cdot 1}, \ldots, a_{0 \cdot m_0}, a_{1 \cdot 0}, a_{1 \cdot 1},  \ldots, a_{1 \cdot m_1}, \ldots \rangle.\]
We call this two-sided asymptotic $\delta$-pseudo-orbit $\varphi$. In order to simplify notation we now denote the $k$\textsuperscript{th} coordinate of $\varphi$ by $a_k$, so that, for example, $a_0 = a_{0 \cdot 0}$ is the $0$\textsuperscript{th} coordinate of $\varphi$ and $a_{-1}= a_{1 \cdot m_1}$ is the $(-1)$\textsuperscript{th} coordinate of $\varphi$. With this revised notation $\varphi = \langle a_i \rangle _{i \in \mathbb{Z}}$. From the construction of $\varphi$ it follows that
\[A=\bigcap_{n\geq 0} \overline{\{a_i \mid i \geq n\}},\]
 and
 \[A=\bigcap_{n\leq 0} \overline{\{a_i \mid i \leq n\}}.\]

Let $\langle x_i\rangle_{i \in \mathbb{Z}}$ be a full trajectory such that $d(x_i, a_i) < \nicefrac{\epsilon}{2}$ for all $i \in \mathbb{Z}$.  We claim that $d_H(\alpha(\langle x_i\rangle), A) <\epsilon$. Indeed, pick $a \in A$. Then there is a decreasing sequence $\langle i_n\rangle_{n \in \mathbb{N}}$ of negative integers such that $a= \lim_{n \to \infty} a_{i_n}$. Thus there is $N \in \mathbb{N}$ such that $d(a, a_{i_n})< \nicefrac{\epsilon}{3}$ for all $n>N$. 
Since $d(x_{i_n}, a_{i_n}) <\nicefrac{\epsilon}{2}$ for all $n \in \mathbb{N}$, it follows that $x_{i_n} \in B_\frac{5\epsilon}{6}(a)$ for $n>N$. By compactness the sequence $\langle x_{i_n}\rangle_{n>N}$ has a limit point $z \in \overline{B_{\nicefrac{5\epsilon}{6}}(a)}$: in particular $d(z,a) < \epsilon$. 
Hence $z \in \alpha(\langle x_i\rangle)$ and
\begin{equation}\label{eqn1}
    A \subseteq \bigcup_{y \in \alpha(\langle x_i\rangle)} B_\epsilon(y).
\end{equation} 

Now take $z \in \alpha(\langle x_i\rangle)$. Then there is a decreasing sequence $\langle i_n\rangle_{n \in \mathbb{N}}$ of negative integers such that $z=\lim_{n \to \infty} x_{i_n}$. Let $k \in \mathbb{N}$ be such that $d(z, x_{i_k})< \nicefrac{\epsilon}{2}$. By shadowing $d(a_{i_k} , x_{i_k})< \nicefrac{\epsilon}{2}$. By the triangle inequality $d(z,a_{i_k}) < \epsilon$. Since $a_{i_k} \in A$ it follows that 
\begin{equation}\label{eqn2} \alpha(\langle x_i\rangle) \subseteq \bigcup_{a \in A} B_\epsilon(a).
\end{equation}
By Equations (\ref{eqn1}) and (\ref{eqn2}) it follows that $d_H(\alpha(\langle x_i\rangle_{i \in \mathbb{Z}}), A) <\epsilon$. 

The fact that $d_H(\omega(x_0), A) <\epsilon$ follows by similar argument. 


\end{proof}

The following example shows that the converse to Theorem \ref{thmShadICT} is false.

\begin{example}\label{ExampleConverseThmICTShadFalse}
Define a map $f\colon [-1,1] \to [-1,1]$ by
\[f(x)=\left\{\begin{array}{lll}
(x+1)^2-1& \text{if} & x \in [-1,0),
\\x^2 & \text{if} & x \in [0,1] .
\end{array}\right.\]
Then $f$ does not have shadowing but $ICT_f=\alpha_f=\omega_f$.
\end{example}

\begin{figure}[h]
\centering
\begin{tikzpicture}[scale=2.25]
\datavisualization [school book axes,
                    visualize as smooth line,
                    y axis={label},
                    x axis={label} ]

data [format=function] {
      var x : interval [-1:0] samples 100;
      func y = -1+ (1+\value{x})^2;
      }
data [format=function] {
      var x : interval [0:1] samples 100;
      func y = (\value x)^2;
      };
\end{tikzpicture}
\caption{Example \ref{ExampleConverseThmICTShadFalse}}
\label{figureES}
\end{figure}

In Example \ref{ExampleConverseThmICTShadFalse}, it is easy to see that $ICT_f=\alpha_f=\omega_f =\{\{-1\},\{0\},\{1\}\}$. However $f$ does not have shadowing. Let $\epsilon= 1/3$. For any $\delta>0$ we can construct a $\delta$-pseudo-orbit which is not $\epsilon$-shadowed. Indeed, fix $\delta>0$ and let $n >1$ be such that $\nicefrac{1}{n} < \delta$. Now pick $z \in (\nicefrac{2}{3}, 1)$ such that $\nicefrac{1}{n} \in \Orb^+ _f(z)$. Let $m \in \mathbb{N}$ be such that $f^m(z)=\nicefrac{1}{n}$. Now let $k \in \mathbb{N}$ be such that $f^k(-\nicefrac{1}{n}) \in (-1,-\nicefrac{3}{4})$. Then \[\langle z, f(z), \ldots , f^m(z), 0, -\nicefrac{1}{n}, f(-\nicefrac{1}{n}), \ldots , f^k(-\nicefrac{1}{n})\rangle\] is a finite $\delta$-pseudo orbit. Suppose $x$ $\epsilon$-shadows this pseudo-orbit. Then $x \in B_\epsilon(z) \subseteq (\nicefrac{1}{3}, 1]$. But $[0,1]$ is strongly invariant under $f$, hence $\Orb^+ _f(x) \subseteq [0,1]$. Since $(-1,-\nicefrac{3}{4}) \cap B_\epsilon([0,1]) = \emptyset$ this is a contradiction: $f$ does not exhibit shadowing. 

\medskip



\begin{corollary}\label{CorollaryClosuresEqual}
Let $(X,f)$ be a dynamical system with shadowing. Then $\overline{\alpha_f}=\overline{\omega_f}=ICT_f$.
\end{corollary}

\begin{proof}
By Lemma \ref{LemmaAlphaOmegaAreICT}, $ICT_f \supseteq \alpha_f$ and $ICT_f \supseteq \omega_f$. The result now follows immediately from Theorem \ref{thmShadICT}.
\end{proof}

\begin{remark}
The fact that $\overline{\omega_f}=ICT_f$ for systems with shadowing has been proved previously by Meddaugh and Raines in \cite{MeddaughRaines}.
\end{remark}

Since $ICT_f$ is always closed in the hyperspace $2^X$ (see Lemma \ref{lemmaICTclosed}), we also get the following corollary.
\begin{corollary}\label{CorollaryClosed} Let $(X,f)$ be a dynamical system for which $\alpha_f=ICT_f$. Then $\alpha_f$ is closed.
\end{corollary}

Theorem \ref{thmShadICT} suggests the following question: when is it the case that every element of $ICT_f$ is both the $\alpha$-limit set and the $\omega$-limit set of the same full trajectory? The next result gives a sufficient condition for this to be the case.

\begin{theorem}\label{thmSLimShadICT} Let $(X,f)$ be a dynamical system with two-sided s-limit shadowing. Then for any $A \in ICT_f$ there is a full trajectory $\langle x_i \rangle_{i \in \mathbb{Z}}$ such that $\alpha( \langle x_i \rangle)= \omega(\langle x_i \rangle ) =A$.
In particular $\alpha_f=\omega_f = ICT_f$.
\end{theorem}

\begin{proof}
Let $A \in ICT_f$ and let $\epsilon>0$ be given. By two-sided s-limit shadowing there exists $\delta>0$ such that every two-sided asymptotic $\delta$-pseudo-orbit is asymptotically $\epsilon/2$-shadowed by a full trajectory (without loss of generality we assume $\delta < \nicefrac{\epsilon}{2}$). 

Now follow the construction of the two-sided asymptotic $\delta$-pseudo orbit $\langle a_i \rangle _{i \in \mathbb{Z}}$ in the proof of Theorem \ref{thmShadICT}. Recall that 
\[A=\bigcap_{n\geq 0} \overline{\{a_i \mid i \geq n\}},\]
and
\[A=\bigcap_{n\leq 0} \overline{\{a_i \mid i \leq n\}}.\]
Let $\langle x_i\rangle_{i \in \mathbb{Z}}$ be a full trajectory such that 
\begin{enumerate}
    \item $d(x_i, a_i) < \epsilon/2$ for all $i \in \mathbb{Z}$,
    \item $\lim_{i \to \pm \infty} d(x_i, a_i) =0$. 
\end{enumerate}
It follows that $\alpha( \langle x_i \rangle)= \omega(\langle x_i \rangle ) =A$. The fact that $\alpha_f=\omega_f = ICT_f$ now follows from Lemma \ref{LemmaAlphaOmegaAreICT}.
\end{proof}

\begin{remark}
 We did not use the fact that $\langle x_i\rangle_{i \in \mathbb{Z}}$ $\nicefrac{\epsilon}{2}$-shadows $\langle a_i\rangle_{i \in \mathbb{Z}}$ in the proof of Theorem \ref{thmSLimShadICT}. Therefore, we could replace the hypothesis of ``two-sided s-limit shadowing" with the weaker condition: ``there exists $\delta>0$ such that for any two-sided $\delta$-pseudo-orbit $\langle y_i \rangle_{i \in \mathbb{Z}}$ there exists a full trajectory $\langle z_i \rangle_{i \in \mathbb{Z}}$ such that $\lim_{i \to \pm \infty} d(y_i, z_i) =0$." 
\end{remark}

A system $(X,f)$ is \emph{expansive} if there exists $\eta>0$ (referred to as an \emph{expansivity constant}) such that given any two distinct full trajectories $\langle x_i \rangle _{i \in \mathbb{Z}}$ and $\langle y_i \rangle _{i \in \mathbb{Z}}$ there exists $i \in \mathbb{Z}$ such that $d(x_i, y_i) \geq \eta$. In \cite{BarwellGoodOprochaRaines} the first author \emph{et al.}\ showed that an expansive map has shadowing if and only if it has s-limit shadowing. We extended that result in \cite{GoodMaciasMeddaughMitchellThomas} to show that an expansive map has shadowing if and only if it has two-sided s-limit shadowing. Combining this result with Theorem \ref{thmSLimShadICT}, we immediately obtain the following.

\begin{theorem}\label{thmShadExpansiveICT} Let $(X,f)$ be a dynamical system with shadowing. If $f$ is expansive then for any $A \in ICT_f$ there is a full trajectory $\langle x_i \rangle_{i \in \mathbb{Z}}$ such that $\alpha( \langle x_i \rangle)= \omega(\langle x_i \rangle ) =A$.
In particular $\alpha_f=\omega_f = ICT_f$.
\end{theorem}

\begin{corollary}\label{CorollarySFTAlphaEqualsOmegaEqualsICT}
Let $(X, \sigma)$ be a shift of finite type (whether one- or two- sided). Then for any $A \in ICT_\sigma$ there is a full trajectory $\langle x_i \rangle_{i \in \mathbb{Z}}$ such that $\alpha( \langle x_i \rangle)= \omega(\langle x_i \rangle ) =A$.
In particular $\alpha_\sigma=\omega_\sigma = ICT_\sigma$.
\end{corollary}
\begin{proof}
Shifts of finite type are precisely the shift systems that exhibit shadowing \cite{Walters}. By Theorem \ref{thmShadExpansiveICT} it now suffices to note that all shift spaces are expansive.
\end{proof}

\begin{remark}
 Corollary \ref{CorollarySFTAlphaEqualsOmegaEqualsICT} enhances a result of Barwell \emph{et al.}\ \cite{BarwellGoodKnightRaines} who show that $ICT_\sigma = \omega_\sigma$ for shifts of finite type.
\end{remark}

\subsection{A remark on $\gamma$-limit sets}
At this point we digress from our main topic to make a brief foray into $\gamma$-limit sets. First introduced by Hero \cite{Hero} who studied them for interval maps, $\gamma$-limit sets have since been further examined by Sun \emph{et al.}\ in \cite{Sun} and \cite{Sun2} for graph maps and dendrites respectively. The $\gamma$-\emph{limit set} of a point $x$, denoted $\gamma(x)$, is defined by saying that, for any $y \in X$, $y \in \gamma (x)$ if and only if $y \in \omega(x)$ and there exists a sequence $\langle y_i\rangle_{i=1} ^\infty$ in $X$ and a strictly increasing sequence $\langle n_i\rangle_{i=1} ^\infty$ in $\mathbb{N}$ such that $f^{n_i}(y_i)=x$ for each $i$ and $\lim_{i \to \infty} y_i = y$. Note that it is possible that $\gamma(x) = \emptyset$. 
We denote by $\gamma_f$ the set of all $\gamma$-limit sets of $(X,f)$.

\begin{remark}
Whilst we have refrained from defining the $\alpha$-limit set of a point, if one were to use Hero's definition of such (see Section \ref{SectionAlphaLimitTypes}), then it would follow that $\gamma(x) = \alpha(x) \cap \omega(x)$.
\end{remark}

\begin{remark}\label{RemarkGammaLimitSetsForHomeomorphisms}
For a dynamical system $(X,f)$, if $f$ is a homeomorphism it is easy to see that, for any $x \in X$, $\gamma(x)= \alpha(\langle x_i \rangle ) \cap \omega(x)$, where $\langle x_i \rangle _{i \leq 0}$ is the unique backward trajectory of $x$.
\end{remark}





Unlike $\alpha$- and $\omega$- limit sets, $\gamma$-limit sets 
are not necessarily internally chain transitive. The example below demonstrates this.

\begin{example}\label{ExampleGammaLimitSetNotICT}
Let $(X, \sigma)$ be the full two-sided shift with alphabet $\{0,1,2\}$. Consider the point $x$:
\[ x= \ldots 0^n1^n 0^{n-1}1^{n-1}\ldots 0^2 1^ 201\cdot 0^221^220^321^3 \ldots 0^n 2 1^n \ldots .\]
Then $\gamma(x)$ is not internally chain transitive.
\end{example}
In Example \ref{ExampleGammaLimitSetNotICT}, let $\langle x_i \rangle _{i \leq 0}$ be the unique backward trajectory of $x$. By Theorem \ref{thmLimitSetsInTwoSidedShiftSpace} we can observe that:
\[ \alpha(\langle x_i \rangle)= \{0^\infty, 1^\infty, \sigma^n(0^\infty \cdot 1^\infty) \mid n \in \mathbb{Z}\},\]
\[ \omega(x)= \{0^\infty, 1^\infty, \sigma^n(0^\infty 2\cdot 1^\infty) \mid n \in \mathbb{Z}\}.\]
Since $\sigma$ is a homeomorphism, by Remark \ref{RemarkGammaLimitSetsForHomeomorphisms},
\[\gamma(x)= \{0^\infty, 1^\infty\}.\]
It is obvious that $\gamma(x)$ is not internally chain transitive.

\medskip

Example \ref{ExampleGammaLimitSetNotICT} notwithstanding, every $\gamma$-limit set is closed and contained in a single \emph{chain component} of the dynamical system, i.e.\ for each $\epsilon>0$ and for all $a,b\in\gamma(x)$ there is an $\epsilon$-chain from $a$ to $b$ in $X$ (as opposed to in $\gamma(x)$).

\begin{proposition}
Let $(X,f)$ be a dynamical system. For any $x \in X$, $\gamma(x)$ is closed and contained in a single chain component of $(X,f)$.
\end{proposition}

\begin{proof}
If $\gamma(x) = \emptyset$ then the closedness holds and chain transitivity is vacuous.

Let $a,b \in \gamma(x)$. Let $\delta>0$ be given. Let $y \in X$ be such that $d(f(y), f(a)) < \delta$ and there exists $n > 1$ such that $f^n(y)=x$: such a point exists by the continuity of $f$ combined with the fact that $a \in \gamma(x)$. Now let $m \in \mathbb{N}$ be such that $d(f^m(x), b) < \delta$. It follows that $\langle a, f(y), f^2(y) \ldots , f^n(y)=x, f(x), f^2(x),\ldots , f^{m-1}(x), b\rangle$ is a $\delta$-chain from $a$ to $b$.

Now suppose $z \in \overline{\gamma(x)}$. Then there is a sequence $\langle y_i \rangle _{i=1} ^\infty$ in $\gamma(x)$ such that $\lim_{i \to \infty} y_i =z$. Note that, since $\omega(x)$ is closed and $y_i \in \omega(x)$ for each $i$ it follows that $z \in \omega(x)$. Now, for each $i \in \mathbb{N}$, let $z_i \in B_{\nicefrac{1}{i}}(y_i)$ and $n_i \in \mathbb{N}$ be such that $f^{n_i}(z_i)=x$ and $\langle n_i \rangle _{i=1} ^\infty$ is an increasing sequence. Then, as  $\lim_{i \to \infty} z_i =z$, it follows that $z \in \gamma(x)$.
\end{proof}

Using theorems \ref{thmSLimShadICT} and \ref{thmShadExpansiveICT} we obtain the following corollaries concerning the nonempty closed internally chain transitive sets in systems with two-sided s-limit shadowing.


\begin{corollary}
If $(X,f)$ is a system with two-sided s-limit shadowing then $ICT_f \subseteq \gamma_f$.
\end{corollary}
\begin{proof}
Let $A \in ICT_f$. By Theorem \ref{thmSLimShadICT} there is a full trajectory $\langle x_i \rangle _{i \in \mathbb{Z}}$ through $x_0=x$ such that $\alpha( \langle x_i \rangle )= \omega(x)= A$. Notice that $\gamma(x) \subseteq \omega(x)$ by definition. Since $\alpha(\langle x_i \rangle)=\omega(x)$, and $\langle x_i \rangle _{i\leq 0}$ is a backward trajectory of $x$, it follows that $\gamma(x)=A$. Hence $ICT_f \subseteq \gamma_f$.
\end{proof}

\begin{corollary}\label{CorollaryShadExpansiveGammaICT}
If $(X,f)$ is an expansive system with shadowing then $ICT_f \subseteq \gamma_f$.
\end{corollary}

\section{Characterising $\overline{\alpha_f}=ICT_f$ and $\alpha_f=ICT_f$}\label{SectionCharacterise}

In \cite{GoodMeddaugh2016} the authors characterise systems for which $\overline{\omega_f}=ICT_f$ and $\omega_f=ICT_f$ in terms of novel shadowing properties. In this section we show that the natural backward analogues of these shadowing properties characterise when $\overline{\alpha_f}=ICT_f$ and $\alpha_f=ICT_f$. We also demonstrate by way of examples that, in contrast to the shadowing property, there is no general entailment between the backward and forward versions of these types of shadowing. 

In \cite{GoodMeddaugh2016} it is shown that the property of $\overline{\omega_f}=ICT_f$ is characterised by a variation on shadowing the authors term \emph{cofinal orbital shadowing}. A system $f \colon X \to X$ has the cofinal orbital shadowing property if for all $\epsilon>0$ there exists $\delta>0$ such that for any $\delta$-pseudo-orbit $\langle x_i \rangle _{i = 0} ^\infty$ there exists a point $ z \in X$ such that for any $K \in \mathbb{N}$ there exists $N \geq K$ such that 
\[d_H(\overline{\{f^{N+i}(z)\}_{i =0} ^\infty} , \overline{\{x_{N+i}\}_{i = 0} ^\infty})<\epsilon.\]
The authors additionally demonstrate that this form of shadowing is equivalent to one which seems \emph{prima facie} stronger: the \emph{eventual strong orbital shadowing property}. A system $f \colon X \to X$ has the eventual strong orbital shadowing property if for all $\epsilon>0$ there exists $\delta>0$ such that for any $\delta$-pseudo-orbit $\langle x_i \rangle _{i = 0} ^\infty$ there exists a point $z \in X$ and $K \in \mathbb{N}$ such that 
\[d_H(\overline{\{f^{N+i}(z)\}_{i = 0} ^\infty} , \overline{\{x_{N+i}\}_{i = 0} ^\infty})<\epsilon\]
for all $N \geq K$.

\begin{definition}
A system $f \colon X \to X$ has the \emph{backward cofinal orbital shadowing property} if for all $\epsilon>0$ there exists $\delta>0$ such that for any backward $\delta$-pseudo-orbit $\langle x_i \rangle _{i \leq 0}$ there exists a backward trajectory $\langle z_i \rangle _{i\leq 0}$ such that for any $K \in \mathbb{N}$ there exists $N \geq K$ such that 
\[d_H(\overline{\{ z_{i-N}\} _{i\leq 0}}, \overline{\{x_{i-N}\}_{i \leq 0}})<\epsilon.\]
\end{definition}

\begin{definition}
A system $f \colon X \to X$ has the \emph{backward eventual strong orbital shadowing property} if for all $\epsilon>0$ there exists $\delta>0$ such that for any backward $\delta$-pseudo-orbit $\langle x_i \rangle _{i \leq 0}$ there exists a backward trajectory $\langle z_i \rangle _{i\leq 0}$ and there exists $K \in \mathbb{N}$ such that 
\[d_H(\overline{\{ z_{i-N}\} _{i\leq 0}}, \overline{\{x_{i-N}\}_{i \leq 0}})<\epsilon\]
for all $N \geq K$.
\end{definition}

\begin{theorem}\label{thmCharClosureAlphaEqualsICT}
Let $(X,f)$ be a dynamical system. The following are equivalent: 
\begin{enumerate}
    \item $f$ has the backward cofinal orbital shadowing property;
    \item $f$ has the backward eventual strong orbital shadowing property;
    \item $\overline{\alpha_f} = ICT_f$.
\end{enumerate}
\end{theorem}

\begin{proof}
From the definitions it is easy to see that $(2) \implies (1)$. We will show $(1) \implies (3)$ and that $(3) \implies (2)$.

Suppose that $f$ has the backward cofinal orbital shadowing property. Recall that $\overline{\alpha_f} \subseteq ICT_f$, hence it will suffice to show $ICT_f \subseteq \overline{\alpha_f}$. Let $A \in ICT_f$. Let $\epsilon>0$ be given. It will suffice to show there exists $B \in \alpha_f$ with $d_H(A,B) < \epsilon$. Let $\delta>0$ correspond to $\nicefrac{\epsilon}{2}$ for cofinal orbital shadowing. Now, follow the construction of the sequence $\langle a_i \rangle _{i \in \mathbb{Z}}$ in Theorem \ref{thmShadICT} (but for $\nicefrac{\epsilon}{2}$ and $\delta$ as here) and let $x_i=a_i$ for all $i\leq 0$. Recall that this means
\[ A= \alpha(\langle x_i \rangle _{i \leq 0}).\]
Let $\langle z_i \rangle _{i\leq 0}$ be given by backward cofinal orbital shadowing so that for any $K \in \mathbb{N}$ there exists $N \geq K$ such that
\[d_H(\overline{\{ z_{i-N}\} _{i\leq 0}}, \overline{\{x_{i-N}\}_{i \leq 0}})<\nicefrac{\epsilon}{2}.\]
Notice that in particular this means that
\[ d_H(\alpha(\langle x_i \rangle),\alpha(\langle z_i \rangle))< \epsilon.\]
Since $\alpha(\langle x_i \rangle _{i \leq 0})= A$ it follows that $A \in \overline{\alpha_f}$.

\medskip

Now suppose that $(X,f)$ does not have backward eventual strong orbital shadowing and let $\epsilon>0$ witness this. (We will show $ICT_f \neq \alpha_f$.) This means that for each $n \in \mathbb{N}$ there is a backward $\nicefrac{1}{2^n}$-pseudo-orbit $\langle x^n _i \rangle _{i \leq 0}$ such that for any backward orbit $\langle z_i \rangle _{i\leq 0}$ and any $K \in \mathbb{N}$ there exists $N \geq K$ with
\[d_H(\overline{\{ z_{i-N}\} _{i\leq 0}}, \overline{\{x_{i-N}\}_{i \leq 0}})\geq \epsilon.\]
It follows that, in particular, for each backward orbit $\langle z_i \rangle _{i\leq 0}$ and any $n \in \mathbb{N}$
\begin{equation}\label{eqn1ThmcharClosure} d_H(\alpha(\langle z_i \rangle),\alpha(\langle x^n _i \rangle) )\geq \nicefrac{\epsilon}{2}. \end{equation}
For each $n \in \mathbb{N}$ let $A_n =\alpha(\langle x^n_i \rangle _{i \leq 0})$. The sequence of compact sets $\langle A_n \rangle _{n \in \mathbb{N}}$ has a convergent subsequence which converges in the hyperspace $2^X$. Without loss of generality we may assume the sequence itself is convergent; let $A$ be its limit. We claim $A \in ICT_f$ but that $A \notin \overline{\alpha_f}$.

Let $a,b \in A$ and let $\xi >0$ be arbitrary. By the uniform continuity of $f$, there exists $\eta >0$ such that for any $x,y \in X$ if $d(x,y) < \eta$ then $d(f(x),f(y)) < \nicefrac{\xi}{2}$. Without loss of generality take $\eta < \nicefrac{\xi}{2}$. Let $M \in \mathbb{N}$ be such that $\nicefrac{1}{2^M}<\nicefrac{\eta}{3}$ and $d_H(A_M,A)<\nicefrac{\eta}{3}$. Now take $K \in \mathbb{N}$ such that
\[d_H(\overline{\{x^M_{i-K}\}_{i\leq 0}}, A_M) <\nicefrac{\eta}{3}.\]
Thus
\[d_H(\overline{\{x^M_{i-K}\}_{i\leq 0}}, A) <\nicefrac{2\eta}{3}.\]
Let $m \in \mathbb{N}$ be such that $d(x^M_{-m-K}, b)< \nicefrac{2\eta}{3}$ and let $l>m$ be such that $d(x^M_{-l-K}, a)< \nicefrac{2\eta}{3}$. Let $y_0=a$ and $y_{l-m}=b$. For each $j \in {1, \ldots, l-m-1}$ pick $y_j \in A$ with $d(y_j, x^M _{-l-K+j})< \nicefrac{2\eta}{3}$. We claim $\langle y_0, y_1 , \ldots , y_{l-k} \rangle$ is a $\xi$-chain from $a$ to $b$. Indeed, for $j \in \{0, \ldots , l-m-1\}$
\begin{align*}
    d(f(y_j),y_{j+1}) &\leq d(f(y_j), f(x^M _{-l-K+j}))+ d(f(x^M _{-l-K+j}), x^M _{-l-K+j +1}) 
    \\ &\text{      }\hspace{0.3cm}+ d(x^M _{-l-K+j +1}, y_{j+1})
    \\ &\leq \nicefrac{\xi}{2} + \nicefrac{1}{2^M} + \nicefrac{2\eta}{3}
    \\ &\leq \nicefrac{\xi}{2} + \nicefrac{\eta}{3} + \nicefrac{2\eta}{3}
    \\ & \leq \xi.
\end{align*}
Since $a$ and $b$ were chosen arbitrarily in $A$ we have that $A$ is internally chain transitive. Thus, since $A$ is nonempty and closed, $A \in ICT_f$.

Suppose for a contradiction that $A \in \overline{\alpha_f}$. Then there exists a backward trajectory $\langle z_i \rangle _{i\leq 0} \in X$ such that $d_H(\alpha(\langle z_i \rangle _{i\leq 0}), A) < \nicefrac{\epsilon}{4}$. Let $M \in \mathbb{N}$ be such that $d_H(A_M, A) < \nicefrac{\epsilon}{4}$. Then $d_H(\alpha(\langle z_i \rangle _{i\leq 0}), A_M) < \nicefrac{\epsilon}{2}$, which contradicts Equation \ref{eqn1ThmcharClosure}. Therefore $A \in ICT_f \setminus \overline{\alpha_f}$. Thus $\overline{\alpha_f}\neq ICT_f$.
\end{proof}

\begin{remark}
Unlike with shadowing (see Lemma \ref{lemmaBackwardsShadowing}), none of the shadowing properties in Theorem \ref{thmCharClosureAlphaEqualsICT} imply their forward analogues (nor vice-versa). To see this, by Theorem \ref{thmCharClosureAlphaEqualsICT} and \cite[Theorem 13]{GoodMeddaugh2016}, it suffices to give an example where $\alpha_f=ICT_f$ but $\overline{\omega_f} \neq ICT_f$ and an example where $\omega_f=ICT_f$ but $\overline{\alpha_f} \neq ICT_f$. Examples \ref{Example1} and \ref{Example2} provide this.
\end{remark}

\begin{example}\label{Example1} Let
$x=1010^210^3\ldots$.
Take
\[ X = \overline{\Orb^+ _\sigma(x) \cup \{ 0^nx \mid n \in \mathbb{N}\}} ,\]
where the closure is taken with regard to the one-sided full shift on the alphabet $\{0,1\}$, and consider the system $(X,\sigma)$. Then $ICT_\sigma=\omega_\sigma \neq \overline{\alpha_\sigma}$.
\end{example} 
In Example \ref{Example1}, $\omega(x)=\{0^\infty, 0^n10^\infty \mid n \geq 0\}$. It is easy to see that the only other $\omega$-limit set is $\{0^\infty\}$. Thus \[\omega_\sigma=\{\{0^\infty\}, \{0^\infty, 0^n10^\infty \mid n \geq 0\}\}.\] Meanwhile \[\alpha_\sigma=\{\{0^\infty\}\}.\] Observe that $\overline{\alpha_\sigma}=\alpha_\sigma$. Finally $ICT_\sigma=\omega_\sigma \neq \overline{\alpha_\sigma}$. Hence the system has cofinal orbital shadowing and eventual strong orbital shadowing by \cite[Theorem 13]{GoodMeddaugh2016} but the system does not have their backward analogues by Theorem \ref{thmCharClosureAlphaEqualsICT}.

\begin{example}\label{Example2}
Take
\[ X = \overline{\{ \sigma^k(10^n10^{n-1}\ldots 10^{2}10^\infty) \mid k,n \in \mathbb{N}\}},\]
where the closure is taken with regard to the one-sided full shift on the alphabet $\{0,1\}$, and consider the system $(X,\sigma)$. Then $ICT_\sigma=\alpha_\sigma \neq \overline{\omega_\sigma}$.
\end{example}
In Example \ref{Example2} it is easily observed that \[\omega_\sigma=\{\{0^\infty\}\}.\] Meanwhile \[\alpha_\sigma=\{\{0^\infty\}, \{ 0^\infty, 0^n10^\infty \mid n \geq 0\}\}.\] Observe that $\overline{\omega_\sigma}=\omega_\sigma$. Finally $ICT_\sigma=\alpha_\sigma \neq \overline{\omega_\sigma}$. Hence the system $(X,\sigma)$ has backward cofinal orbital shadowing and backward eventual strong orbital shadowing by Theorem \ref{thmCharClosureAlphaEqualsICT} but it does not have their forward analogues by \cite[Theorem 13]{GoodMeddaugh2016}.

\medskip

In \cite[Theorem 22]{GoodMeddaugh2016} the authors show that the property of $\omega_f=ICT_f$ is characterised by several equivalent asymptotic variants of shadowing: These are \emph{asymptotic orbital shadowing}, \emph{asymptotic strong orbital shadowing} and \emph{orbital limit shadowing}. The system $(X,f)$ has then has the \emph{asymptotic orbital shadowing} property if for any asymptotic pseudo-orbit $\langle x_i \rangle_{i\geq 0}$ there exists a point $z\in X$ such that for any $\epsilon>0$ there exists $N \in \mathbb{N}$ such that 
\[
		d_H(\overline{\{x_{N+i}\}_{i\geq 0}},\overline{\{f^{N+i}(z)\}_{i\geq 0}})<\epsilon.
\]
The system has the \emph{asymptotic strong orbital shadowing} property if for any asymptotic pseudo-orbit $\langle x_i \rangle_{i\geq 0}$ there exists a point $z\in X$ such that for any $\epsilon>0$ there exists $K \in \mathbb{N}$ such that 
\[
		d_H(\overline{\{x_{N+i}\}_{i\geq 0}},\overline{\{f^{N+i}(z)\}_{i\geq 0}})<\epsilon
\]
for all $N \geq K$. Finally, the system has the \emph{orbital limit shadowing} property, as introduced by Pilyugin \cite{Pilyugin2007}, if for any asymptotic pseudo-orbit $\langle x_i \rangle_{i\geq 0}$ there exists a point $z\in X$ such that $\omega(z) = \omega(\langle x_i \rangle).$

Before characterising $\omega_f=ICT_f$ by these notions of shadowing, the authors \cite{GoodMeddaugh2016} note that \emph{asymptotic shadowing}, also known as limit shadowing, is sufficient but not necessary for $\omega_f=ICT_f$: a system has asymptotic shadowing if for each asymptotic pseudo-orbit $\langle x_i \rangle _{i \geq 0}$ there exists a point $z \in X$ such that 
\[ \lim_{ i \rightarrow \infty} d(f^i(z), x_i) =0.\]
As with other shadowing variants, asymptotic shadowing has a backward analogue.
\begin{definition}
A system $f \colon X \to X$ has the \emph{backward asymptotic shadowing property} if for each backward asymptotic pseudo-orbit $\langle x_i \rangle _{i \leq 0}$ there exists a backward trajectory $\langle z_i \rangle _{i \leq 0}$ such that 
\[ \lim_{ i \rightarrow -\infty} d(z_i, x_i) =0.\]
\end{definition}
We shall see (Corollary \ref{CorollaryBackAsymSHad}) that backward asymptotic shadowing is sufficient for $\alpha_f=ICT_f$, however it is not necessary. The irrational rotation of the circle satisfies $\alpha_f=ICT_f$ (as a minimal map, both are equal to \{$X$\}) however it fails to have backward asymptotic shadowing. To see this one can observe that for any irrational rotation $f$ of the circle, the inverse function $f^{-1}$ is also an irrational rotation of the circle. It thereby suffices to note that no irrational rotation of the circle has asymptotic shadowing \cite{Pilyugin2007}.


\begin{definition}
A system $f \colon X \to X$ has the \emph{backward asymptotic orbital shadowing property} if for each backward asymptotic pseudo-orbit $\langle x_i \rangle _{i \leq 0}$ there exists a backward trajectory $\langle z_i \rangle _{i \leq 0}$ such that for any $\epsilon>0$ there exists $N \in \mathbb{N}$ such that 
\[ d_H(\overline{\{z_{i-N}\}_{i \leq 0}}, \overline{\{x_{i-N}\}_{i \leq 0}}) <\epsilon.\]
\end{definition}

\begin{definition}
A system $f \colon X \to X$ has the \emph{backward asymptotic strong orbital shadowing property} if for each backward asymptotic pseudo-orbit $\langle x_i \rangle _{i \leq 0}$ there exists a backward trajectory $\langle z_i \rangle _{i \leq 0}$ such that for any $\epsilon>0$ there exists $K \in \mathbb{N}$ such that 
\[ d_H(\overline{\{z_{i-N}\}_{i \leq 0}}, \overline{\{x_{i-N}\}_{i \leq 0}}) <\epsilon\]
for all $N \geq K$.
\end{definition}

The following is a backward version of the orbital limit shadowing property, studied by Pilyugin \emph{et al.}\ \cite{Pilyugin2007}.

\begin{definition}
A system $f \colon X \to X$ has the \emph{backward orbital limit shadowing property} if for each backward asymptotic pseudo-orbit $\langle x_i \rangle _{i \leq 0}$ there exists a backward trajectory $\langle z_i \rangle _{i \leq 0}$ such that 
\[\alpha(\langle z_i \rangle)=\alpha(\langle x_i \rangle).\]
\end{definition}

As mentioned previously, Hirsch \emph{et al.}\ \cite{Hirsch} showed that the $\alpha$-limit set (resp.\ $\omega$-limit set) of any backward (resp.\ forward) pre-compact trajectory is internally chain transitive. In the same paper, the authors show that the $\omega$-limit set of any pre-compact asymptotic pseudo-orbit is internally chain transitive \cite[Lemma 2.3]{Hirsch}. Whilst we omit the proof, the same is true of pre-compact backward asymptotic pseudo-orbits. We formulate this as Lemma \ref{LemmaOmegaAlphaAsymPseudoAreICT} below.


\begin{lemma}\textup{\cite{Hirsch}}\label{LemmaOmegaAlphaAsymPseudoAreICT}
Let $(X,f)$ be a dynamical system where $X$ is a (not necessarily compact) metric space. The $\alpha$-limit set (resp.\ $\omega$-limit set) of any backward (resp.\ forward) pre-compact asymptotic pseudo-orbit is internally chain transitive. In particular, when $X$ is compact, all such limit sets are in $ICT_f$.
\end{lemma}

\begin{theorem}\label{thmCharAlphaEqualsICT}
Let $(X,f)$ be a dynamical system. The following are equivalent:
\begin{enumerate}
    \item $\alpha_f =ICT_f$;
    \item $f$ has the backward orbital limit shadowing property;
    \item $f$ has the backward asymptotic orbital shadowing property;
    \item $f$ has the backward asymptotic strong orbital shadowing property.
\end{enumerate}
\end{theorem}

\begin{proof}
Clearly $(4) \implies (3)$. It is also easy to see that $(2) \implies (4)$. We will show $(3) \implies (1) \implies (2)$.

To this end, suppose that $f$ has backward asymptotic orbital shadowing. Let $A \in ICT_f$. Form a backward asymptotic pseudo-orbit $\langle x_i \rangle _{i \leq 0}$ by following the construction as in the proof of Theorem \ref{thmShadICT} and taking $x_i=a_i$ for all $i \leq 0$. (We may ignore the $\epsilon$ and $\delta$ in the construction, we can simply take $l=0$.) Recall that this means
\[A=\bigcap_{n\leq 0} \overline{\{x_i \mid i \leq n\}},\]
or equivalently,
\[ A= \alpha(\langle x_i \rangle).\]
Let $\langle z_i \rangle _{i\leq 0}$ be given by backward asymptotic orbital shadowing. Now let $\epsilon>0$ be given and let $N \in \mathbb{N}$ be such that 
\[d_H(\alpha(\langle z_i \rangle ),\overline{\{ z_{i-N}\} _{i\leq 0}})<\nicefrac{\epsilon}{3},\]
\[d_H(\overline{\{ z_{i-N}\} _{i\leq 0}}, \overline{\{x_{i-N}\}_{i \leq 0}})<\nicefrac{\epsilon}{3},\]
and
\[d_H(\overline{\{ x_{i-N}\} _{i\leq 0}}, \alpha(\langle x_i \rangle ))<\nicefrac{\epsilon}{3}.\]
By the triangle inequality it follows that $d_H(\alpha(\langle z_i \rangle ), A) <\epsilon$. Since $\epsilon>0$ was picked arbitrarily this implies that $A=\alpha(\langle z_i \rangle )$. Hence $ICT_f \subseteq \alpha_f$. By Lemma \ref{LemmaAlphaOmegaAreICT} we have $\alpha_f \subseteq ICT_f$, thus $(1)$ holds.

Now suppose that $\alpha_f=ICT_f$. Let $\langle x_i \rangle _{i \leq 0}$ be a backward asymptotic pseudo-orbit. By Lemma \ref{LemmaOmegaAlphaAsymPseudoAreICT} $\alpha(\langle x_i\rangle) \in ICT_f$. Since $\alpha_f = ICT_f$ there exists a backward trajectory $\langle z_i \rangle _{i \leq 0}$ with $\alpha(\langle z_i\rangle) = \alpha( \langle x_i \rangle )$. Hence $f$ has the backward orbital limit shadowing property, i.e.\ $(2)$ holds. 
\end{proof}

\begin{corollary}\label{CorollaryBackAsymSHad}
If $(X,f)$ has backward asymptotic shadowing then $\alpha_f=ICT_f$.
\end{corollary}
\begin{proof}
By Theorem \ref{thmCharClosureAlphaEqualsICT} it suffices to note that backward asymptotic shadowing implies backward orbital limit shadowing.
\end{proof}

\begin{remark} Combining theorems \ref{thmCharClosureAlphaEqualsICT} and \ref{thmCharAlphaEqualsICT} we have that if $\alpha_f$ is closed then the following are equivalent:
\begin{enumerate}
    \item $f$ has the backward orbital limit shadowing property;
    \item $f$ has the backward eventual strong orbital shadowing property;
    \item $f$ has the backward asymptotic (strong) orbital shadowing property;
    \item $f$ has the backward cofinal orbital shadowing property.
\end{enumerate}
\end{remark}

\begin{remark}
Examples \ref{Example1} and \ref{Example2}, together with Theorem \ref{thmCharAlphaEqualsICT} and \cite[Theorem 22]{GoodMeddaugh2016}, show that, unlike shadowing (see Lemma \ref{lemmaBackwardsShadowing}), neither the backward orbital limit shadowing property nor the backward asymptotic orbital shadowing nor the backward asymptotic strong orbital shadowing is equivalent to its forward analogue.
\end{remark}


\section*{References}

\bibliographystyle{plain} 
\bibliography{bib}

\begin{thebibliography}{10}

\bibitem{AgronskyBrucknerCederPearson}
S.~J. Agronsky, A.~M. Bruckner, J.~G. Ceder, and T.~L. Pearson.
\newblock The structure of {$\omega$}-limit sets for continuous functions.
\newblock {\em Real Anal. Exchange}, 15(2):483--510, 1989/90.

\bibitem{BalibreaPiotr}
Francisco Balibrea, Gabriela Dvorn\'{\i}kov\'{a}, Marek Lampart, and Piotr
  Oprocha.
\newblock On negative limit sets for one-dimensional dynamics.
\newblock {\em Nonlinear Anal.}, 75(6):3262--3267, 2012.

\bibitem{BarwellGoodKnightRaines}
Andrew Barwell, Chris Good, Robin Knight, and Brian~E. Raines.
\newblock A characterization of {$\omega$}-limit sets in shift spaces.
\newblock {\em Ergodic Theory Dynam. Systems}, 30(1):21--31, 2010.

\bibitem{BarwellGoodOprocha}
Andrew~D. Barwell, Chris Good, and Piotr Oprocha.
\newblock Shadowing and expansivity in subspaces.
\newblock {\em Fund. Math.}, 219(3):223--243, 2012.

\bibitem{BarwellGoodOprochaRaines}
Andrew~D. Barwell, Chris Good, Piotr Oprocha, and Brian~E. Raines.
\newblock Characterizations of {$\omega$}-limit sets in topologically
  hyperbolic systems.
\newblock {\em Discrete Contin. Dyn. Syst.}, 33(5):1819--1833, 2013.

\bibitem{BarwellMeddaughRaines2015}
Andrew~D. Barwell, Jonathan Meddaugh, and Brian~E. Raines.
\newblock Shadowing and {$\omega$}-limit sets of circular {J}ulia sets.
\newblock {\em Ergodic Theory Dynam. Systems}, 35(4):1045--1055, 2015.

\bibitem{BarwellRaines2015}
Andrew~D. Barwell and Brian~E. Raines.
\newblock The {$\omega$}-limit sets of quadratic {J}ulia sets.
\newblock {\em Ergodic Theory Dynam. Systems}, 35(2):337--358, 2015.

\bibitem{BlokhBrucknerHumkeSmital}
Alexander Blokh, A.~M. Bruckner, P.~D. Humke, and J.~Sm\'{\i}tal.
\newblock The space of {$\omega$}-limit sets of a continuous map of the
  interval.
\newblock {\em Trans. Amer. Math. Soc.}, 348(4):1357--1372, 1996.

\bibitem{bowen-markov-partitions}
R.~Bowen.
\newblock Markov partitions for {A}xiom {${\rm A}$} diffeomorphisms.
\newblock {\em Amer. J. Math.}, 92:725--747, 1970.

\bibitem{Bowen}
Rufus Bowen.
\newblock {$\omega $}-limit sets for axiom {${\rm A}$} diffeomorphisms.
\newblock {\em J. Differential Equations}, 18(2):333--339, 1975.

\bibitem{brian-oprocha}
Will Brian and Piotr Oprocha.
\newblock Ultrafilters and {R}amsey-type shadowing phenomena in topological
  dynamics.
\newblock {\em Israel J. Math.}, 227(1):423--453, 2018.

\bibitem{bmr}
William~R. Brian, Jonathan Meddaugh, and Brian~E. Raines.
\newblock Chain transitivity and variations of the shadowing property.
\newblock {\em Ergodic Theory Dynam. Systems}, 35(7):2044--2052, 2015.

\bibitem{BrucknerSmital}
Andrew~M. Bruckner and Jaroslav Sm\'{\i}tal.
\newblock The structure of {$\omega$}-limit sets for continuous maps of the
  interval.
\newblock {\em Math. Bohem.}, 117(1):42--47, 1992.

\bibitem{Corless}
Robert~M. Corless.
\newblock Defect-controlled numerical methods and shadowing for chaotic
  differential equations.
\newblock {\em Phys. D}, 60(1-4):323--334, 1992.
\newblock Experimental mathematics: computational issues in nonlinear science
  (Los Alamos, NM, 1991).

\bibitem{CorlessPilyugin}
Robert~M. Corless and S.~Yu. Pilyugin.
\newblock Approximate and real trajectories for generic dynamical systems.
\newblock {\em J. Math. Anal. Appl.}, 189(2):409--423, 1995.

\bibitem{Coven1}
Ethan~M. Coven, Ittai Kan, and James~A. Yorke.
\newblock Pseudo-orbit shadowing in the family of tent maps.
\newblock {\em Trans. Amer. Math. Soc.}, 308(1):227--241, 1988.

\bibitem{Coven}
Ethan~M. Coven and Zbigniew Nitecki.
\newblock Nonwandering sets of the powers of maps of the interval.
\newblock {\em Ergodic Theory Dynamical Systems}, 1(1):9--31, 1981.

\bibitem{CuiDing}
Hongfei Cui and Yiming Ding.
\newblock The {$\alpha$}-limit sets of a unimodal map without homtervals.
\newblock {\em Topology Appl.}, 157(1):22--28, 2010.

\bibitem{Dastjerdi}
Dawoud~Ahmadi Dastjerdi and Maryam Hosseini.
\newblock Sub-shadowings.
\newblock {\em Nonlinear Anal.}, 72(9-10):3759--3766, 2010.

\bibitem{deVries}
Jan de~Vries.
\newblock {\em Topological dynamical systems}, volume~59 of {\em De Gruyter
  Studies in Mathematics}.
\newblock De Gruyter, Berlin, 2014.
\newblock An introduction to the dynamics of continuous mappings.

\bibitem{Fakhari}
Abbas Fakhari and F.~Helen Ghane.
\newblock On shadowing: ordinary and ergodic.
\newblock {\em J. Math. Anal. Appl.}, 364(1):151--155, 2010.

\bibitem{GoodMaciasMeddaughMitchellThomas}
Chris Good, Sergio Macías, Jonathan Meddaugh, Joel Mitchell, and Joe Thomas.
\newblock Expansivity and unique shadowing, 2020.

\bibitem{GoodMeddaugh2016}
Chris Good and Jonathan Meddaugh.
\newblock Orbital shadowing, internal chain transitivity and {$\omega$}-limit
  sets.
\newblock {\em Ergodic Theory Dynam. Systems}, 38(1):143--154, 2018.

\bibitem{GoodMeddaugh2018}
Chris Good and Jonathan Meddaugh.
\newblock Shifts of finite type as fundamental objects in the theory of
  shadowing.
\newblock {\em Inventiones mathematicae}, Dec 2019.

\bibitem{GoodMitchellThomas2}
Chris Good, Joel Mitchell, and Joe Thomas.
\newblock On inverse shadowing, 2019.

\bibitem{GoodMitchellThomas}
Chris Good, Joel Mitchell, and Joe Thomas.
\newblock Preservation of shadowing in discrete dynamical systems.
\newblock {\em J. Math. Anal. Appl.}, 485(1):123767, 39, 2020.

\bibitem{GoodOprochaPuljiz2019}
Chris Good, Piotr Oprocha, and Mate Puljiz.
\newblock Shadowing, asymptotic shadowing and s-limit shadowing.
\newblock {\em Fund. Math.}, 244(3):287--312, 2019.

\bibitem{Hero}
Michael~W. Hero.
\newblock Special {$\alpha$}-limit points for maps of the interval.
\newblock {\em Proc. Amer. Math. Soc.}, 116(4):1015--1022, 1992.

\bibitem{Hirsch}
Morris~W. Hirsch, Hal~L. Smith, and Xiao-Qiang Zhao.
\newblock Chain transitivity, attractivity, and strong repellors for
  semidynamical systems.
\newblock {\em J. Dynam. Differential Equations}, 13(1):107--131, 2001.

\bibitem{Lee}
Keonhee Lee.
\newblock Continuous inverse shadowing and hyperbolicity.
\newblock {\em Bull. Austral. Math. Soc.}, 67(1):15--26, 2003.

\bibitem{LeeSakai}
Keonhee Lee and Kazuhiro Sakai.
\newblock Various shadowing properties and their equivalence.
\newblock {\em Discrete Contin. Dyn. Syst.}, 13(2):533--540, 2005.

\bibitem{MaiShao}
Jie-Hua Mai and Song Shao.
\newblock Spaces of {$\omega$}-limit sets of graph maps.
\newblock {\em Fund. Math.}, 196(1):91--100, 2007.

\bibitem{MeddaughRaines}
Jonathan Meddaugh and Brian~E. Raines.
\newblock Shadowing and internal chain transitivity.
\newblock {\em Fund. Math.}, 222(3):279--287, 2013.

\bibitem{Mitchell}
Joel Mitchell.
\newblock Orbital shadowing, {$\omega$}-limit sets and minimality.
\newblock {\em Topology Appl.}, 268:106903, 7, 2019.

\bibitem{Nusse}
Helena~E. Nusse and James~A. Yorke.
\newblock Is every approximate trajectory of some process near an exact
  trajectory of a nearby process?
\newblock {\em Comm. Math. Phys.}, 114(3):363--379, 1988.

\bibitem{Oprocha2016}
Piotr Oprocha.
\newblock Shadowing, thick sets and the {R}amsey property.
\newblock {\em Ergodic Theory Dynam. Systems}, 36(5):1582--1595, 2016.

\bibitem{Pearson}
D.~W. Pearson.
\newblock Shadowing and prediction of dynamical systems.
\newblock {\em Math. Comput. Modelling}, 34(7-8):813--820, 2001.

\bibitem{Pennings}
Timothy Pennings and Jeffrey Van~Eeuwen.
\newblock Pseudo-orbit shadowing on the unit interval.
\newblock {\em Real Anal. Exchange}, 16(1):238--244, 1990/91.

\bibitem{PiluginRodSakai2002}
S.~Yu. Pilyugin, A.~A. Rodionova, and K.~Sakai.
\newblock Orbital and weak shadowing properties.
\newblock {\em Discrete Contin. Dyn. Syst.}, 9(2):287--308, 2003.

\bibitem{Pilyugin}
Sergei~Yu. Pilyugin.
\newblock {\em Shadowing in dynamical systems}, volume 1706 of {\em Lecture
  Notes in Mathematics}.
\newblock Springer-Verlag, Berlin, 1999.

\bibitem{Pilyugin2007}
Sergei~Yu. Pilyugin.
\newblock Sets of dynamical systems with various limit shadowing properties.
\newblock {\em J. Dynam. Differential Equations}, 19(3):747--775, 2007.

\bibitem{Pokluda}
David Pokluda.
\newblock On the transitive and {$\omega$}-limit points of the continuous
  mappings of the circle.
\newblock {\em Arch. Math. (Brno)}, 38(1):49--52, 2002.

\bibitem{robinson-stability}
Clark Robinson.
\newblock Stability theorems and hyperbolicity in dynamical systems.
\newblock In {\em Proceedings of the {R}egional {C}onference on the
  {A}pplication of {T}opological {M}ethods in {D}ifferential {E}quations
  ({B}oulder, {C}olo., 1976)}, volume~7, pages 425--437, 1977.

\bibitem{Sakai2003}
Kazuhiro Sakai.
\newblock Various shadowing properties for positively expansive maps.
\newblock {\em Topology Appl.}, 131(1):15--31, 2003.

\bibitem{Sun2}
Taixiang Sun, Yalin Tang, Guangwang Su, Hongjian Xi, and Bin Qin.
\newblock Special {$\alpha$}-limit points and {$\gamma$}-limit points of a
  dendrite map.
\newblock {\em Qual. Theory Dyn. Syst.}, 17(1):245--257, 2018.

\bibitem{Sun}
TaiXiang Sun, HongJian Xi, and HaiLan Liang.
\newblock Special {$\alpha$}-limit points and unilateral {$\gamma$}-limit
  points for graph maps.
\newblock {\em Sci. China Math.}, 54(9):2013--2018, 2011.

\bibitem{Walters}
Peter Walters.
\newblock On the pseudo-orbit tracing property and its relationship to
  stability.
\newblock In {\em The structure of attractors in dynamical systems ({P}roc.
  {C}onf., {N}orth {D}akota {S}tate {U}niv., {F}argo, {N}.{D}., 1977)}, volume
  668 of {\em Lecture Notes in Math.}, pages 231--244. Springer, Berlin, 1978.

\end{thebibliography}

\end{document}